\documentclass[11pt]{article}
\usepackage{amssymb}
\usepackage{amsthm}
\usepackage{amsmath}
\usepackage{enumerate}
\usepackage{cite}
\usepackage{indentfirst}
\usepackage{fancyhdr}
\usepackage{graphicx}
\usepackage{color}
\fancypagestyle{preContent}
\usepackage[top=2.54cm,bottom=2.54cm,left=3.18cm,right=3.18cm]{geometry}
\usepackage{fancyhdr}
\newtheorem{Theorem}{Theorem}[section]
\newtheorem{lemma}{Lemma}[section]

\newtheorem{corollary}{Corollary}[section]
\usepackage{hyperref}
\hypersetup{hypertex=true,colorlinks=true,linkcolor=blue,anchorcolor=blue,citecolor=blue}

\begin{document}
\title{ An upper bound on the nullity of signed graphs\footnote{This work was supported by the National Nature Science Foundation of China (Nos.11871040).}}

\author{\centerline{ Keming Liu \footnote{Corresponding author. 
			E-mail address: 1358365098@qq.com(Keming Liu).}
		$\ \ $ Xiying Yuan\footnote{
			E-mail address: xiyingyuan@shu.edu.cn(Xiying Yuan).} }\\[2mm]	
\centerline{\normalsize  Department of Mathematics, Shanghai University, Shanghai 200444, P. R. China.}}
\date{}
\maketitle
{\bf Abstract:}
In this paper, an upper bound on the nullity of signed graphs in terms of  the cyclomatic number and the number of pendant vertices is proved, and the corresponding extremal signed graphs are completely characterized.

\section{Introduction}
All graphs considered in this paper are simple and finite. Let $G=(V(G), E(G))$ be an undirected graph with $V(G)=\left\{v_{1}, v_{2}, \cdots, v_{n}\right\}$ be the vertex set and $E(G)$ be the edge set. We write $x \sim y$ to mean two vertices $x$ and $y$ of $G$ are adjacent. For a vertex $u \in V(G)$, the degree of $u$, denoted by $d_{G}(u)$, is the number of vertices which are adjacent to $u$. A vertex of $G$ is called a pendant vertex (or a leaf) if it is a vertex of degree one in $G$. The number of leaves in $ G $ is denoted by $ p(G) $. Denote by $ P_{n} $ and $C_{n}$ a path and cycle on $n$ vertices, respectively. 
The adjacency matrix $A(G)$ of $G$ is an $n \times n$ matrix whose $(i, j)$-entry equals to $ 1 $ if $v_{i} \sim v_{j}$, and 0 otherwise.  

A signed graph $\varGamma=(G, \sigma)$ consists of a simple graph $G=(V(G), E(G))$, referred to as its underlying graph, and a mapping $\sigma$: $ E(G) \rightarrow\{+,-\}$, its edge labelling. The adjacency matrix of $\varGamma$ is $A(\varGamma)=\left(a_{i j}^{\sigma}\right)$ with $a_{i j}^{\sigma}=\sigma\left(v_{i} v_{j}\right) a_{i j}$, where $\left(a_{i j}\right)$ is the adjacency matrix of the underlying graph $G$. An edge $e$ is said to be positive or negative if $\sigma\left(e\right)=+$ or $\sigma\left(e\right)=-$, respectively. And a simple graph can always be viewed as a singed graph with all positive edges. Let $C$ be a cycle of $\varGamma$, the sign of $C$ is defined by $\sigma(C)=\prod_{e \in C} \sigma(e) $.  A cycle $C$ is said to be positive or negative if $\sigma(C)=+$ or $\sigma(C)=-$, respectively. The rank and nullity of a signed graph $\varGamma$ are  defined to be the rank and the multiplicity of the zero eigenvalues of its adjacency matrix, respectively, written as $r(\varGamma)$ and $ \eta(\varGamma) $, respectively.
For $\lambda \in \mathbb{R}$, we denote by $m(\varGamma, \lambda)$ the algebraic multiplicity of $\lambda$ in the adjacency spectrum of $\varGamma$. If $\lambda$ is not an eigenvalue of $\varGamma$, we put $m(\varGamma, \lambda)=0$. Obvious $ r(\varGamma) + \eta(\varGamma) = |V(\varGamma)| $. 

Let $ \varGamma $ be a signed graph, $ U \subset V(\varGamma) $, $ \varGamma^{U} $  be the signed graph obtained from $ \varGamma $ by reversing the sign of each edge between a vertex in $ U $ and a vertex in $ V(\varGamma) \setminus U $, and $\Gamma-U$ be the signed graph obtained from $ \varGamma $ by removing the vertices of $U$ together with all edges incident to them.
A block of a graph is a maximal subgraph without any cut-vertex. 

The nullity is a classical topic in spectrum theory of graphs. About the nullity of simple graphs and its applications, there are many known results, see [5], [8], [10], [11], [12], [13], [14], [17] and [19] for details.
For signed graph, Liu and You in [18] investigated the nullity of the bicyclic signed, obtained the nullity set of unbalanced bicyclic signed graphs, and determined the nullity set of bicyclic signed graphs. Fan characterized several classes of signed graphs according to nullity in [7]. 
In addition, more results of signed graphs and their applications, see [20], [24] and  [25]. Let $c(G)$ be the cyclomatic number of a graph $G$, that is $c(G)=|E(G)|-|V(G)|+\omega(G)$, where $\omega(G)$ is the number of connected components of $G$. 
For a signed graph $\Gamma$,  the  cyclomatic number and the number of pendant vertices of $\Gamma$ are defined to be the  cyclomatic number and the number of pandent vertices of its underlying graph, respectively. An upper bound on the nullity of graph in terms of $ c(G) $ and $ p(G) $ was proved in [23]. Moreover, the extremal graphs have been characterized in [6], [15] and [23]. 
In this paper, we get an upper bound of the nullity of signed graph in terms of $ c(\Gamma) $ and $ p(\Gamma) $(See Theorem $ 3.1 $), and characterize the corresponding extremal signed graphs( See Section $ 5 $). 

\section{Tools for the nullity of the signed graph}
In this section, useful tools to study the nullity of the signed graph will be listed.
\begin{lemma}\cite{A3} Let $(P_{n}, \sigma)$ be a signed path. Then $ \eta(P_{n}, \sigma)=1 $  if  $ n $  is odd, and $ \eta(P_{n}, \sigma)=0 $  if  $ n $  is even.
\end{lemma}

\begin{lemma} \cite{A7}
	 Let $\left(C_{n}, \sigma\right)$ be a signed cycle. Then $\eta\left(C_{n}, \sigma\right)=2$ if and only if  $\left(C_{n}, \sigma\right)$ is posotive and $n \equiv 0 (mod\ 4)$ or $\left(C_{n}, \sigma\right)$ is negative and $n \equiv 2 (mod\ 4)$, $\eta\left(C_{n}, \sigma\right)=0$ otherwise.
\end{lemma}

\begin{lemma} \cite{A2}
	Let $ v $ be a vertex of  $\Gamma$, then $ \eta(\Gamma)-1 \leq \eta(\Gamma-v) \leq \eta(\Gamma)+1 $. 
\end{lemma}
For an induced subgraph $\Gamma_{1}$ of $\Gamma$ and a vertex $x$ outside $\Gamma_{1}$, the induced subgraph of $\Gamma$ with vertex set $V\left(\Gamma_{1}\right) \cup\{x\}$ is simply written as $\Gamma_{1}+x,$  and sometimes we use the notation $\Gamma-\Gamma_{1}$ instead of $\Gamma-V\left(\Gamma_{1}\right)$.
\begin{lemma} \cite{A18}
Let $v$ be a cut-vertex of a connected signed graph $\Gamma$, and $\Gamma_{1}$ be a component of $\Gamma-v$. 

(1) If $\eta\left(\Gamma_{1}\right)=\eta\left(\Gamma_{1}+v\right)-1$, then $\eta(\Gamma)=\eta\left(\Gamma_{1}\right)+\eta\left(\Gamma-\Gamma_{1}\right)$; 

(2)	If $\eta\left(\Gamma_{1}\right)=\eta\left(\Gamma_{1}+v\right)+1$, then $ \eta(\Gamma)=\eta(\Gamma-v)-1 $. 
\end{lemma}
\begin{lemma} \cite{A26}
Let $G$ be a graph and $v \in V(G)$, then we have $c(G-v)=c(G)$ if $v$ does not lie on a cycle of $G$, otherwise $c(G-v) \leq c(G)-1$.	
\end{lemma}

A matching $M$ in $G$ is a set of pairwise non-adjacent edges, that is, no two edges share a common vertex. A maximum matching is a matching that contains the largest possible number of edges. The cardinality of a maximum matching is called the matching number of $ G $, denoted by $ \mu(G) $. 
For a tree $ T $ on at least two vertices, a vertex $ u \in T $ is called a covered vertex  in $ T $, if all maximum matchings of $ T $ that cover $ u $; otherwise, $ u $ is not a covered vertex in $ T $.

\begin{lemma} \cite{A7}
	Let $\Gamma$ be a signed graph of order $n$ and $T$ be a signed tree. Denote by $T(u) \odot^{k} \Gamma$ a signed graph which is obtained from $ T \cup U $ by joining $u \in V(T)$ and arbitrary $k(1 \leq k \leq n)$ vertices of $\Gamma$.
	
	$ (1) $ If $u$ is a covered vertex in $T$, then $\eta\left(T(u) \odot^{k} \Gamma\right)=\eta(T)+\eta(\Gamma)$.
	
	$ (2) $ If $u$ is not a covered vertex in $T$, then $\eta\left(T(u) \odot^{k} \Gamma\right)=\eta(T)-1+\eta(\Gamma+u)$.
\end{lemma}
\begin{lemma} 
Let $\Gamma$ be a signed graph with $V(\Gamma)=\left\{v_{1}, v_{2}, \cdots, v_{n}\right\}$. Denote by $\tilde{\Gamma}$ the signed graph obtained from $\Gamma$ by replacing each $v_{i}$ of $\Gamma$ with an independent set of $m_{i}$ vertices $v_{i}^{1}, v_{i}^{2}, \cdots, v_{i}^{m_{i}}$ and $ v_{i}^{x} $ $ ^{\pm}_{\sim} $ $ v_{j}^{y} $ in $\tilde{\Gamma}$ if and only if $ v_{i} $ $ ^{\pm}_{\sim} $ $ v_{j} $ in $\Gamma$, where $x \in\left\{1,2, \cdots, m_{i}\right\}$ and $y \in\left\{1,2, \cdots, m_{j}\right\}$. Then $r(\Gamma)=r(\tilde{\Gamma})$.
\end{lemma}
\begin{proof}
	Let $\tilde{A}$ and $ A $ be the adjacency matrix of $\tilde{\Gamma}$ and $\Gamma$, respectively. It is easy to get $ r(\tilde{A}) = r(A) $ by performing elementary row transformation on $\tilde{A}$. 
\end{proof}
\begin{lemma} \cite{A7}
	Let $ v $ be a pendant vertex of $\Gamma$ and $ u $ be the neighbour of $v$. Then 
	$ \eta(\Gamma)=\eta(\Gamma-v-u ) $. 
\end{lemma}

We have the following Lemma by Lemma $ 2.3 $ and Lemma $ 2.8 $.

\begin{lemma} 
 Let $\Gamma$ be a signed graph obtained from a signed path $P_{m}$ with $m \geq 2$ and a disjoint signed graph $H$ by identifying a pendant vertex of $P_{m}$ with a vertex of $H$. Then $m(\Gamma, \lambda) \leq m(H, \lambda)+1$ for any $\lambda \in \mathbb{R}$.
\end{lemma}
\begin{lemma} \cite{A21}
	Let $\Gamma$ be a signed graph obtained from a signed cycle $C_{s}$ and a disjoint signed graph $H$ which are joined by a signed path $P_{m}$ with $m \geq 1$, i.e., identify one pendant vertex of $P_{m}$ with a vertex of $C_{s}$ and another with a vertex of $H$. Then $m(\Gamma, \lambda) \leq m(H, \lambda)+2$ for any $\lambda \in \mathbb{R}$.
\end{lemma}
\begin{lemma} \cite{A21}
	Let $\Gamma$ be a simple signed graph obtained from a signed graph $H$ and a disjoint signed path $P_{m}$ with $m \geq 2$ by identifying two pendant vertices of $P_{m}$ respectively with two distinct vertices of $H$. Then $m(\Gamma, \lambda) \leq m(H, \lambda)+2$ for any $\lambda \in \mathbb{R}$.
\end{lemma}
\section{An upper bound on the nullity of signed graph}

In this section, we mainly get an upper bound of the nullity $ \eta(\Gamma) $ of signed graph $ \Gamma $ in terms of the cyclomatic number $ c(\Gamma) $ and the number of pendant vertices $ p(\Gamma) $.

If a neighbor $y$ of $x \in V(G)$ has degree 2, i.e., $d(y)=2$, it is called a 2-degree neighbor of $x$. By using the similar arguments of the proof of Lemma $ 2.6$ in [23], we have the following result.
\begin{lemma}
	Let $\Gamma$ be a connected signed graph with a vertex $x$. Let $r$, $ m $ be the number of components containing 2-degree neighbors of $x$ and the number of 2 -degree neighbors of $x$, respectively, and $ s=\omega(\Gamma-x) $. Then

	$ (1) $ $d(x)+r \geq m+s$;
	
	$ (2) $ $2 d(x)+r \geq m+2 s+1$, if $x$ lies on a cycle of $\Gamma$;
	
$ 	(3) $ $c(\Gamma-x)=c(\Gamma)-d(x)+s$.
\end{lemma}
	We call $\Gamma$ a cycle-disjoint signed graph if any two distinct cycles of $\Gamma$ (if any) have no common vertices. Now, we give an upper bound  on the nullity of signed graph.

	\begin{Theorem}
		Let $\Gamma$ be a  signed graph in which every component has at least two vertices. Then
		$$
		\eta(\Gamma) \leq\left\{\begin{array}{cc}
		2 c(\Gamma)+p(\Gamma)-1, & \text { if } p(\Gamma)\geq1 \text { ; } \\
		2 c(\Gamma), & \text { if } p(\Gamma)=0 \text { and } \Gamma \text { is a cycle-disjoint signed graph; } \\
		2 c(\Gamma)-1, & \text { if } p(\Gamma)=0 \text { and some cycles have common vertices. }
		\end{array}\right.
		$$\end{Theorem}
	\begin{proof}
		We may assume that $\Gamma$ is connected. We proceed by induction on $ n $, the order of $\Gamma$ to prove the inequality. If $ n=2 $, the results is obvious. Now we assume that $n \geq 3$ and 
		the inequality holds for any connected signed graph $\Gamma$ with  order $2 \leq|V(\Gamma)| \leq n-1$. The proof is divided into three cases according to the parameter $ p(\Gamma) $.

		\textbf{ Case 1.}  $ p(\Gamma)=0 $.
		
		In this case, $\Gamma$ contains cycles. Let $x$ be a vertex lying on a cycle, and let $\Gamma - x=H_{1} \cup \ldots \cup H_{s}$ be the connected components of $ \Gamma - x $.

		Without loss of generality, let $H_{i}$ be the components containing 2-degree neighbors of $x$ for $i=1, \ldots, r$, and $H_{j}$ be the components containing no 2-degree neighbors of $x$ for $j=r+1, \ldots, s$. This arrangement implies that each $H_{i}$ 
		has pendant vertices, and each $H_{j}$ has no pendant vertices. Since $\Gamma$ has no pendant vertices, $2 \leq\left|V\left(H_{i}\right)\right|<|V(\Gamma)|$, then we have 
		$$ \eta\left(H_{i}\right) \leq 2 c\left(H_{i}\right)+p\left(H_{i}\right)-1 \text { for } i=1,2, \ldots, r; $$ 
		$$ 	\eta\left(H_{j}\right) \leq 2 c\left(H_{j}\right) \text { for } j=r+1, \ldots, s. $$
		by applying induction hypothesis to each component of $\Gamma-x$.
		
		\leftline {By Lemma $ 2.3 $, we have}
		$$ 	\eta(\Gamma) \leq 1+\eta(\Gamma-x)=1+\sum_{i=1}^{r} \eta\left(H_{i}\right)+\sum_{j=r+1}^{s} \eta\left(H_{j}\right). $$
		So
		$$
		\begin{aligned}
		\eta(\Gamma) & \leq 1+\sum_{i=1}^{r}\left(2 c\left(H_{i}\right)+p\left(H_{i}\right)-1\right)+\sum_{j=r+1}^{s} 2 c\left(H_{j}\right) \\
		&=1-r+2 \sum_{i=1}^{s} c\left(H_{i}\right)+\sum_{i=1}^{r} p\left(H_{i}\right),
		\end{aligned}
		$$
		that is 
		$$ \eta(\Gamma) \leq 1-r+2 c(\Gamma-x)+p(\Gamma-x). $$
		As $\Gamma$ has no pendant vertices, the number of pendant vertices of $\Gamma-x$ equals the number of 2-degree neighbors of $x$ in $\Gamma$, that is
		$$ p(\Gamma-x)=\sum_{i=1}^{r} p\left(H_{i}\right)=m, $$
		where $m$ denotes the number of 2-degree neighbors of $x$, especially, if all the components of $ \Gamma-x $  containing no $ 2 $-degree neighbors of $x$, then $  m=0 $. Lemma $3.1$ says that
		$$ c(\Gamma-x)=c(\Gamma)-d(x)+s. $$
		Then we have
		\begin{equation}
		\eta(\Gamma) \leq 2 c(\Gamma)-[2 d(x)-2 s+r-m-1] . \tag{1}
		\end{equation}
		Applying $ (2) $ of Lemma $ 3.1 $, we have $\eta(\Gamma) \leq 2 c(\Gamma)$.
		
		If two distinct cycles of $\Gamma$, say $C_{1}, C_{2}$, have common vertices, let $x$ be a common vertex of $C_{1}$ and $C_{2}$ such that $N_{C_{1}}(x) \neq N_{C_{2}}(x)$, then $d(x) \geq s+2$. By $ (1) $ of Lemma $3.1$, we have
		\begin{equation}
		2 d(x)+r \geq 2 s+m+2  . \tag{2}
		\end{equation}
		Substituting (2) to (1), we have $\eta(\Gamma) \leq 2 c(\Gamma)-1$.
		
		
		
		In the following, we assume $p(\Gamma) \geq 1$.
		
		\textbf{ Case 2.} $p(\Gamma) = 1$.
		
		Let $x$ be the pendant vertex and  $ y $ be the neighbour of $ x $. By Lemma $2.8$, we have 
		$\eta(\Gamma)=\eta\left(\Gamma-x-y\right) $, and $c(\Gamma)=c\left(\Gamma-x-y\right)$, 
		
		If  $ d(y)=2 $, let $ z $ be the other neighbour of $ y $. When $ d(z)=2 $, we have  $ p\left(\Gamma-x-y\right)=p(\Gamma)=1$ and $ c\left(\Gamma-x-y\right)=c(\Gamma) $. So, $ \eta\left(\Gamma-x-y\right) \leq 2 c\left(\Gamma-x-y\right)+p(\Gamma-x-y)-1 $ by applying induction hypothesis to $\Gamma-x-y$ (noting that $\left|V\left(\Gamma-x-y\right)\right| \geq 2$ since $\Gamma$ is assumed with at least 4 vertices). 
		When $ d(z)>2$, we have $p\left(\Gamma-x-y\right)=p(\Gamma)-1=0$. By applying induction hypothesis, we have
		
		$$\eta(\Gamma)=\eta\left(\Gamma-x-y\right) \leq 2 c\left(\Gamma-x-y\right)=2 c(\Gamma)+p(\Gamma)-1.$$
		
		Now, only the case when $p(\Gamma)=1$ and $d(y) \geq 3$ is left. 
		
		Let $\Gamma-y=H_{1} \cup \ldots \cup H_{s}$
		be the connected components of $\Gamma-y$.
		
		Without loss of generality, let $H_{i}$ be the components containing 2-degree neighbors of $y$ for $i=1, \ldots, r$, and $H_{j}$ be the components containing no 2-degree neighbors of $y$ for $j=r+1, \ldots, s$, and $H_{s}$ contains the unique pendant vertex $x$. This arrangement implies that each $H_{i}$ (for $i=1, \ldots, r$ ) has pendant vertices, and each $H_{j}$ (for $j=r+1, \ldots, s-1$ ) has at least two vertices and has no pendant vertices. The induction hypothesis to $H_{i}$ (for $i=1, \ldots, s-1$ ) allows us to assume 
		$$ \eta\left(H_{i}\right) \leq 2 c\left(H_{i}\right)+p\left(H_{i}\right)-1 \text { for } i=1,2, \ldots, r; $$ 
		$$ 	\eta\left(H_{j}\right) \leq 2 c\left(H_{j}\right) \text { for } j=r+1, \ldots, s-1. $$
		From $\eta(\Gamma)=\eta(\Gamma-x-y)$ we have
		$$ \eta(\Gamma)=\sum_{i=1}^{r} \eta\left(H_{i}\right)+\sum_{j=r+1}^{s-1} \eta\left(H_{j}\right). $$
		So 
		$$ \eta(\Gamma) \leq \sum_{i=1}^{r}\left(2 c\left(H_{i}\right)+p\left(H_{i}\right)-1\right)+\sum_{j=r+1}^{s-1} 2 c\left(H_{j}\right), $$
		from which we have
		$$ 	\eta(\Gamma) \leq-r+2 \sum_{i=1}^{s-1} c\left(H_{i}\right)+\sum_{i=1}^{r} p\left(H_{i}\right). $$
		Since $x$ is the unique pendant vertex of $\Gamma$, the number of pendant vertices of $\Gamma-y$ equals the number of 2 -degree neighbors of $y$ in $\Gamma$. Thus
		$$ 	p(\Gamma-y)=\sum_{i=1}^{r} p\left(H_{i}\right)=m, $$ 
		where $m$ denotes the number of 2-degree neighbors of $y$ in $\Gamma$, especially, if all the components of $ \Gamma-y $  containing no $ 2 $ -degree neighbors of $y$, then $ m=0 $. Observing that $c\left(H_{s}\right)=0$ and applying Lemma $3.1$, we have
		$$
		\sum_{i=1}^{s-1} c\left(H_{i}\right)=\sum_{i=1}^{s} c\left(H_{i}\right)=c(\Gamma-y)=c(\Gamma)-d(y)+s.
		$$
		Then we have
		
		\begin{equation}
			\eta(\Gamma) \leq 2 c(\Gamma)-[2 d(y)+r-2 s-m] .    \tag{3}
		\end{equation}
		Lemma $3.1$ says $d(y)+r-s-m \geq 0$. Clearly, $d(y) \geq s$. Thus
		\begin{equation}
			2 d(y)+r-2 s-m \geq 0 .    \tag{4}
		\end{equation}
		Combining $(3)$ and $(4)$, we have $\eta(\Gamma) \leq 2 c(\Gamma)$. As $p(\Gamma)=1$, we obtain the required inequality $\eta(\Gamma) \leq 2 c(\Gamma)+p(\Gamma)-1.$
		
		\textbf{ Case 3.} $p(\Gamma) \geq 2$.
		
		Let $x$ be a pendant vertex and $ y $ be a neighbour of $ x $. By Lemma $2.3$ we have
		$$ \eta(\Gamma) \leq \eta(\Gamma-x)+1. $$
		As $\Gamma-x$ has pendant vertices, the induction hypothesis to $\Gamma-x$ says
		$$ \eta(\Gamma-x) \leq 2 c(\Gamma-x)+p(\Gamma-x)-1. $$
		Obviously, $c(\Gamma-x)=c(\Gamma)$. 
		
		If $d(y)>2$, then
		$$ p(\Gamma-x)=p(\Gamma)-1. $$
		So we have $\eta(\Gamma) \leq 2 c(\Gamma)+p(\Gamma)-1$. 
		
		Now suppose $d(y)=2$. Then we have 
		$ \eta(\Gamma)=\eta(\Gamma-x-y) $, $ c(\Gamma-x-y)=c(\Gamma) $, and $ 1 \leq p(\Gamma)-1 \leq p(\Gamma-x-y)  \leq p(\Gamma) $. By using induction hypothesis to $ \Gamma-x-y $, we have 
		$$\eta(\Gamma)=\eta(\Gamma-x-y) \leq 2 c(\Gamma-x-y)+p(\Gamma-x-y)-1=2 c(\Gamma)+p(\Gamma)-1.$$
		
		
		\end{proof}
\section{Auxiliary results for the characterization of the extremal signed graph}


An internal path of $ \Gamma $ is a path whose internal vertices(except end vertices) have degree $ 2 $ and a major vertex of $ \Gamma $ is a vertex with degree at least $ 3 $. Cycle $ C $ is called pendant in $ \Gamma $ if the cycle contains a unique major vertex of $ \Gamma $.

\begin{lemma} 
	$ (1) $ Suppose  $ P_{6}=v_{1}v_{2}v_{3} v_{4} v_{5} v_{6} $  is a path with four internal vertices of degree $ 2 $ in signed graph $\Gamma$. Let $ \Gamma^{'} $ be the signed graph obtained by replacing $ P_{6} $ with a new edge $ v_{1}v_{6} $ with sign $ \sigma(P_{6}) $. Then $ \eta (\Gamma)=\eta (\Gamma^{'}) $.
	
	$ (2) $ Suppose $C_{t} $ is a pendant cycle with nullity $ 2 $ in $ (\Gamma) $. Let $ \Gamma^{'} $ be the signed graph obtained by replacing $C_{t} $ with a  quadrangle with nullity $ 2 $. Then $ \eta (\Gamma)=\eta (\Gamma^{'}) $.
\end{lemma}

\begin{proof}
	$ (1) $ The sign of edge $ v_{i} v_{i+1} $ is denoted by $ \sigma_{i, i+1} $, $ i=1, 2, 3, 4, 5 $. Arranging the vertices of $\Gamma$ such that the adjacency matrix of $\Gamma$  is
	$$ A(\Gamma)=\left(\begin{array}{cccccccc}
	A_{1}& \alpha & \mathbf{0} & \mathbf{0} & \mathbf{0}& \mathbf{0} & \beta  \\ 
	\alpha^{T}& 0 & \sigma_{1, 2} & 0 & 0 & 0 & 0  \\ 
	\mathbf{0}& \sigma_{1, 2} & 0 & \sigma_{2, 3} & 0 & 0 & 0 \\ 
	\mathbf{0}& 0 & \sigma_{2, 3} & 0 & \sigma_{3, 4} & 0 &0  \\ 
	\mathbf{0}& 0 & 0 & \sigma_{3, 4} & 0 & \sigma_{4, 5} & 0  \\ 
	\mathbf{0}& 0 & 0 & 0 & \sigma_{4, 5} & 0 & \sigma_{5, 6}  \\ 
	\mathbf{\beta}^{T}& 0 & 0 & 0 & 0 & \sigma_{5, 6} & 0   \\ 
	
	\end{array}\right).  $$
	
	By performing elementary row transformation on the matrix $ A(\Gamma) $, we have
	$$ B=\left(\begin{array}{cccccccc}
	A_{1}& \alpha & \mathbf{0} & \mathbf{0} & \mathbf{0}& \mathbf{0} & \beta  \\ 
	\alpha^{T}& 0 & 0 & 0 & 0 & 0 & \sigma_{1, 6}^{'}  \\ 
	\mathbf{0}& \sigma_{1, 2} & 0 & \sigma_{2, 3} & 0 & 0 & 0 \\ 
	\mathbf{0}& 0 & \sigma_{2, 3} & 0 & \sigma_{3, 4} & 0 &0  \\ 
	\mathbf{0}& 0 & 0 & \sigma_{3, 4} & 0 & \sigma_{4, 5} & 0  \\ 
	\mathbf{0}& 0 & 0 & 0 & \sigma_{4, 5} & 0 & \sigma_{5, 6}  \\ 
	\mathbf{\beta}^{T}& \sigma_{1, 6}^{'} & 0 & 0 & 0 & 0 & 0   \\  
	\end{array}\right)  $$
	where $ \sigma_{1, 6}^{'}= \sigma_{1, 2}\sigma_{2, 3}\sigma_{3, 4}\sigma_{4, 5}\sigma_{5, 6}$. Then we have $ r(A(\Gamma)) = r(B) $
	
	Write 
	$$ C = \left(\begin{array}{cccccccc}
	A_{1}& \alpha & \mathbf{0} & \mathbf{0} & \mathbf{0}& \mathbf{0} & \beta  \\ 
	\alpha^{T}& 0 & 0 & 0 & 0 & 0 & \sigma_{1, 6}^{'}  \\ 
	\mathbf{\beta}^{T}& \sigma_{1, 6}^{'} & 0 & 0 & 0 & 0 & 0   \\  
	\end{array} \right) $$
	and 
	$$ D = \left(\begin{array}{cccccccc}
	A_{1}& \alpha  & \beta  \\ 
	\alpha^{T}& 0  & \sigma_{1, 6}^{'}  \\ 
	\mathbf{\beta}^{T}& \sigma_{1, 6}^{'} & 0  \\  
	\end{array} \right) $$
	Then $ r(B) = r(C)+ 4$, $ r(C) = r(D)$ and $ D=A(\Gamma^{'}) $, thus $ (1) $ holds. 
	
	
	$ (2) $ The proof of $ (2) $ is similar to $ (1) $.
\end{proof}

For a signed tree $T$ with at least two vertices, it follows from Theorem $3.1$ that $\eta(T) \leq p(T)-1$. When $p(T)=2,\ T$ becomes a path, then by Lemma $ 2.1 $, the equality holds if and only if $T$ is an even path. When $p(T) \geq 3$, by Lemma $ 3.1 $ in [6] and Lemma $ 2.12$, we have the following Lemma which gives a characterization of signed trees with nullity $p(T)-1$.
\begin{lemma}
	Let $T$ be a signed tree with $p(T)(\geq 3)$ leaves. Then $\eta(T)=p(T)-1$ if and only if $T$ satisfies the following two conditions:
	
	(i) All internal paths from a leaf to a major vertex are odd.
	
	(ii) Let $P(u, v)$ be an internal path from any leaf $u$ of $T$ to a major vertex $v$. Then $T^{\prime}=T-(P(u, v)-v)$ is a tree with nullity $p\left(T^{\prime}\right)-1$, and $v$ is a covered vertex of $T^{\prime}$.
\end{lemma}

By slightly modifying the proofs of Lemma $ 3.4 $ in [6], we have the following Lemma for signed graph.
\begin{lemma} 
	Let $\Gamma$ be a connected leaf-free signed graph with $c(\Gamma) \geq 3$. If $\eta(\Gamma)=2 c(\Gamma)-1$, then
	$ \Gamma $  contains a cut-vertex.
\end{lemma}

Let $ H $ and $ K $  be two disjoint signed graphs, the signed graph $ (H, v; K, u) $ be regarded as a coalescence of $ H $ and $ K $, obtained by identifying a vertex $ v $ of $ H $ with a vertex $ u $ of $ K $. 

By slightly modifying the proofs of Lemma 3.1 and Lamma 3.2 in [15], we have the following results for signed graph.

\begin {lemma} 
Let $ \Gamma=(H, v; K, u) $, then $\eta(\Gamma) \leq \eta(K)+\eta(H-v)+1$.
\end {lemma}

\begin {lemma} 
Let $H$ and $K$ be two disjoint signed graphs. If $\Gamma$ is obtained from $H$ and $K$ by connecting one vertex $v$ of $H$ and one vertex $u$ of $K$ with a path $P_{m}(m \geq 2)$, then $\eta(\Gamma) \leq \eta(H)+\eta(K)+1$.
\end {lemma}

For convenience, a signed graph $\Gamma$ will be said to be $2^{+}$-deficient if $\eta(\Gamma) \leq 2 c(\Gamma)+p(\Gamma)-2$, and it will be said to be 1-deficient if $\eta(\Gamma)=2 c(\Gamma)+p(\Gamma)-1$. By slightly modifying the proofs of Lemma $ 3.3 $ and Lemma $ 3.4 $ in [15], we have the following Lemmas for signed graph.

\begin {lemma} 
Let $\Gamma=(H, v; K, u)$, if $K$ is $2^{+}$-deficient, then so is $\Gamma$.
\end {lemma}

\begin {lemma} 
Let $\Gamma=(T, v; H, u)$, where $v$ be a leaf  of a $ 1 $-deficient tree $T$, $ T $ and $ H $ be two disjoint signed graph. Then $\eta(\Gamma)=$ $\eta(H)+p(T)-2$.
\end {lemma}

1-deficient trees have been characterized by Lemma $4.2$, now we characterize 1-deficient unicyclic signed graphs by referring to Lemma $ 3.5 $ in [15].

\begin {lemma} 
Let $\Gamma$ be a unicyclic signed graph with a unique cycle $C_{m}$. Then $\Gamma$ is 1-deficient if and only if $\Gamma$ is obtained from a 1-deficient tree $T$ with at least three vertices by attaching $C_{m}$ with nullity $ 2 $ at a leaf $x$ of $T$.

\end {lemma}

By slightly modifying the proofs of Lemma $ 3.6 $ and Lemma $ 3.7 $ in [15], we give the following two lemmas to study structural properties of $ 1 $-deficient graphs.

\begin {lemma} 
Let $\Gamma$ be a $ 1 $-deficient connected signed graph with $C_{m}$ as a block. Then

$ (1) $ $ \eta(C_{m})=2 $; 

$ (2) $ Either $C_{m}$ is a pendant cycle of $\Gamma$, or $\Gamma=\infty(m, k, 1)$ such that the nullity of each cycle of $\Gamma$ is $ 2 $.

\end {lemma}
\begin {lemma} 
If $\Gamma$ is a $ 1 $-deficient connected signed graph, then the nullity of each cycle $C_{m}$ in $\Gamma$ is $ 2 $. 
\end {lemma}
\section{Characterization of the extremal signed graph}

In this section, we characterize the corresponding extremal signed graph in Theorem $ 3.1 $.

Firstly, we characterize the signed graph with nullity $ 2 c(\Gamma)+p(\Gamma) $.

\begin{Theorem}
	Let $ \Gamma $ be a signed graph in which every component contains at least two vertices, then $\eta(\Gamma)= 2 c(\Gamma)+p(\Gamma)$ if and only if every component of $\Gamma$ is a signed cycle with nullity $ 2$. 
	\end{Theorem}
\begin{proof}
	The sufficiency is obvious.
	
Now we prove the sufficiency part. 

Firstly,
we have $\eta(H) \leq 2 c(H)+p(H)$ for any connected component $H$ of $\Gamma$  by Theorem $ 3.1 $. According to $\eta(\Gamma)=2 c(\Gamma)+p(\Gamma)$, we have $\eta(H)=2 c(H)+p(H)$ for any component $H$ of $\Gamma$. Next, we will prove that every component $H$ of $\Gamma$ is a signed cycle with nullity $ 2 $. For a given component $H$ of $\Gamma$, we know $H$ contains no pendant vertices and distinct signed cycles of $H$ (if any) have no common vertices, otherwise if H has pendant vertices, then $\eta(H)\leq2 c(H)+p(H)-1$ by Theorem $3.1$, a contradiction with $\eta(H)=2 c(H)+p(H)$, Similarly, distinct signed cycles of $H$ (if any) have no common vertices. If we can prove that $H$ is just a unicyclic signed graph, then $H$ must be a signed cycle with nullity $ 2 $( by Lemma $ 2.2 $). For a contradiction, we assume $H$ contains at least two signed cycles.
	
	For $x, y \in V(H)$, we denote by $d(x, y)$ the length of a shortest path between $x$ and $y .$ By $\mathcal{C}_{H}$ denote the set of signed cycles of $H$. Let $C_{1} \neq C_{2} \in \mathcal{C}_{H}$, set
	
	$$
	d\left(C_{1}, C_{2}\right)=\min \left\{d(x, y): x \in V\left(C_{1}\right), y \in V\left(C_{2}\right)\right\}
	$$
	and set
	$$
	d\left(\mathcal{C}_{H}\right)=\max \left\{d\left(C_{1}, C_{2}\right): C_{1} \neq C_{2} \in \mathcal{C}_{H}\right\}.
	$$
  
  	Taking $x \in V\left(C_{1}\right), y \in V\left(C_{2}\right)$ such that $d(x, y)=d\left(C_{1}, C_{2}\right)=d\left(\mathcal{C}_{H}\right)$. Then $x$ is a unique vertex of $C_{1}$ with degree  $ 3 $, and $x$ is a cut-point of $H$. Suppose the size of $C_{1}$ is $l$, that is $C_{1}=C_{l}$, then $H-x$ has $P_{l-1}$ as a component. Denote $H-P_{l-1}$ by $H_{1}$. Then
	$$
	\theta\left(H_{1}\right)=c(H)-1, \ p\left(H_{1}\right)=1,
	$$
	and it is from Theorem $3.1$ that
	$$
	\eta\left(H_{1}\right) \leq 2 c\left(H_{1}\right)+p\left(H_{1}\right)-1=2 c\left(H_{1}\right) .
	$$
	
	\textbf{ Case 1.}  $ l \equiv 0 (mod \ 4) $
	
	If $C_{l}$ is positive, then $\eta\left(P_{l-1}\right)=1=\eta\left(C_{1}\right)-1$
	 , thus we have by Lemma $2.4(1)$ that
	$$
	\eta(H)=\eta\left(P_{l-1}\right)+\eta\left(H_{1}\right),
	$$
	from which it follows that
	$$
	2 c(H)=2\left(c\left(H_{1}\right)+1\right) \leq 1+2 c\left(H_{1}\right),
	$$
	a contradiction.
	
	If $C_{l}$ is negative, then $ \eta\left(P_{l-1}\right)=0 $ and $ \eta\left(C_{1}\right)=1 $. As $ \eta\left(P_{l-1}\right)=\eta\left(C_{1}\right)+1 $ , we have by Lemma $ 2.4(2) $ that 
	$$
	\eta(H)=\eta(H-x)-1=\eta\left(P_{l-1}\right)+\eta\left(H_{2}\right)-1=\eta\left(H_{2}\right),
	$$
	where $H_{2}=H-C_{1}$. Let $z$ be the adjacent vertex of $x$ outside $C_{1}$. If $d(z)=2$, then $z$ is a pendant vertex of $H_{2} ;$ and if $d(z) \geq 3$, then $z$ is not a pendant vertex of $\mathrm{H}_{2}$. In both cases, $p\left(\mathrm{H}_{2}\right) \leq 1$. Applying Theorem $3.1$, we have
	$$
	\eta\left(H_{2}\right) \leq 2 c\left(H_{2}\right) .
	$$
	then we have 
	$$ 2c(H)= \eta(H)=\eta(H_{2}) \leq 2 c(H_{2}) ,$$
	a contradiction to $ c(H)= c(H_{2}) +1 . $
	
\textbf{ Case 2.} $ l \equiv 2 (mod \ 4) $
	
	If $C_{l}$ is positive, then $ \eta(P_{l-1}) = 1 $ and $ \eta(P_{l-1}+x) = 0 $, that is  $ \eta(P_{l-1}) = \eta(P_{l-1}+x) +1 $, thus by Lemma $2.4(2)$ we have
	\begin{equation}
	\eta(H)=\eta(H-x)-1=\eta\left(P_{l-1}\right)+\eta\left(H_{2}\right)-1=\eta\left(H_{2}\right),    \tag{5}
	\end{equation}
	where $H_{2}=H-C_{1}$. Let $z$ be the adjacent vertex of $x$ outside $C_{1}$. If $d(z)=2$, then $z$ is a pendant vertex of $H_{2} ;$ and if $d(z) \geq 3$, then $z$ is not a pendant vertex of $\mathrm{H}_{2}$. In both cases, $p\left(\mathrm{H}_{2}\right) \leq 1$. Applying Theorem $3.1$, we have
	\begin{equation}
	\eta\left(H_{2}\right) \leq 2 \theta\left(H_{2}\right) .    \tag{6}
	\end{equation}
	Combining (5) and (6), we have
	$$
	2 c(H)=\eta(H)=\eta\left(H_{2}\right) \leq 2 c\left(H_{2}\right),
	$$
	a contradiction to $c(H)=c\left(H_{2}\right)+1$.
	
	If $C_{l}$ is negative, then $ \eta(P_{l-1}) = 1 $ and $ \eta(C_{1})=\eta(P_{l-1}+x) = 2$, that is  $ \eta(P_{l-1}) = \eta(C_{1})-1$, thus we have by Lemma $ 2.4(1)$ that 
	$$
	\eta(H)=\eta\left(P_{l-1}\right)+\eta\left(H_{1}\right),
	$$
	from which it follows that 
	$$ 	2 c(H)=\eta(H)=\eta\left(P_{l-1}\right)+\eta\left(H_{1}\right) \leq 1+2 c\left(H_{1}\right), $$
	also a contradiction. 
	
\textbf{ Case 3.} $l$ is odd.
 
 Then by Lemma $2.3$ we have
	\begin{equation}
	\eta(H) \leq \eta(H-x)+1=\eta\left(P_{l-1}\right)+\eta\left(H_{2}\right)+1=\eta\left(H_{2}\right)+1,   \tag{7}
	\end{equation}
	substituting $\eta(H)=2 c(H)=2 c\left(H_{2}\right)+2$ and $\eta\left(H_{2}\right) \leq 2 c\left(H_{2}\right)$ to $(7)$ we have
	$$
	2 c\left(\mathrm{H}_{2}\right)+2 \leq 2 c\left(\mathrm{H}_{2}\right)+1,
	$$
so, we get a contradiction.
	\end{proof}

By slightly modifying the proofs of  Lemma $ 3.2 $ and Corollary $ 3.3 $ in [6], we have the following results.
\begin{lemma} 
Let $\Gamma$ be a connected leaf-free signed graph which is obtained from two signed graphs $ \Gamma_{1} $ and $ \Gamma_{2} $ by identifying  the unique common vertex $ u $. If $\Gamma_{1}$ is a block of $\Gamma$ with $\eta\left(\Gamma_{1}\right) \leq 2 c\left(\Gamma_{1}\right)-2$ and $\Gamma_{2}$ is a signed graph with $c\left(\Gamma_{2}\right) \geq 2$, then $\eta(\Gamma) \leq 2 c(\Gamma)-2$.
\end{lemma}

\begin{corollary} 
	Let $\Gamma$ be a connected leaf-free signed graph with $c(\Gamma) \geq 3$ and $\eta(\Gamma)=2 c(\Gamma)-1$. If $\Gamma$ contains a pendant cycle $C_{t}$, then $ \eta(C_{t})=2 $.
\end{corollary}

 
 A bicyclic graph is a simple connected graph in which the number of edges equals the number of vertices plus one. There are two basic bicyclic graphs: $\infty$-graph and $\theta$-graph. An $\infty$-graph, denoted by $\infty(p, q, l)$, is obtained from two vertex-disjoint cycles $C_{p}$ and $C_{q}$ by connecting one vertex of $C_{p}$ and one of $C_{q}$ with a path $P_{l}$ of length $l-1$ (in the case of $l=1$, identifying the above two vertices); and a $\theta$-graph, denoted by $\theta(p, q, l)$, is a union of three internally disjoint paths $P_{p+1}, P_{q+1}, P_{l+1}$ of length $p, q, l$ respectively, with common end vertices, where $p, q, l \geq 1$ and at most one of them is 1. Observe that any bicyclic graph $G$ is obtained from an $\infty$-graph or a $\theta$-graph (possibly) by attaching trees to some of its vertices.
 
 Now we give the characterization of the leaf-free bicyclic signed graph with nullity $3$.

\begin{Theorem}
Let $ \Gamma $ be a leaf-free bicyclic signed graph. Then $ \eta(\Gamma)=2c(\Gamma)-1=3 $ if and only if  $ \Gamma \cong \infty(l, k, x) $ or $ \theta(l^{'}, x^{'}, k^{'}) $, where the nullity of each cycle of $ \Gamma$ is $ 2 $, and $ x $ is odd, $ l^{'}, x^{'}, k^{'} $ are even. 
\end{Theorem}
\begin{proof}
Since $ \Gamma $ is a leaf-free bicyclic signed graph and $ \eta(\Gamma)=2c(\Gamma)+p(\Gamma)-1=3 $, $ \Gamma $ is $ 1 $-deficient and 
$ \Gamma \cong \infty(l, k, x) $ or $ \theta(l^{'}, x^{'}, k^{'}) $. By Lemma $ 4.10$, we know the nullity of each cycle $ C_{t} $ of $ \Gamma $ is $ 2 $.

\textbf{ Case 1.} $ \Gamma \cong \infty(l, k, x) $. (See Fig. 1)

Observe that $ v_{1} $ is a cut-vertex of $ \Gamma $, and $ \eta(C_{l})-1= \eta(C_{l}-v_{1})= 1 $ by Lemma $ 2.1 $ and $ 2.2 $, set $ \Gamma_{1}= \Gamma- (C_{l}-v_{1}) $, 
then we have $ \eta(\Gamma)= \eta(C_{l}-v_{1})+ \eta(\Gamma_{1}) $( using Lemma $ 2.4(1)$), thus $ \eta(\Gamma_{1})= 2 $. 

If $ x $ is even, then  $ \eta(\Gamma_{1})= \eta(\Gamma_{1}- v_{1}- v_{2})= \cdots = \eta(\Gamma_{1}- P_{x})= 1  $ by applying Lemma $ 2.9$ repeatedly, a contradiction, thus $ x $ is odd. 
\begin{figure}[!h]
	\centering
	\includegraphics{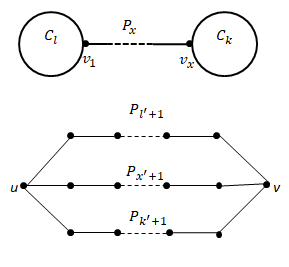}\\[1mm]
	{\small{  \bf Fig. 1~~}$\Gamma$ }
	\label{fig1}
\end{figure}

Conversely, if $ \Gamma \cong \infty(l, k, x) $, where $ \eta(C_{l})= \eta(C_{k})= 2 $ and $ x $ is odd, Observe that $ v_{1} $ is a cut-vertex of $ \Gamma $, and $ \eta(C_{l})-1= \eta(C_{l}-v_{1})= 1 $ by Lemma $ 2.1 $ and $ 2.2 $, set $ \Gamma_{1}= \Gamma- (C_{l}-v_{1}) $, 
then we have $ \eta(\Gamma)= \eta(C_{l}-v_{1})+ \eta(\Gamma_{1}) $( using Lemma $ 2.4(1)$), further $ \eta(\Gamma_{1})= \eta(\Gamma_{1}- v_{1}- v_{2})= \cdots = \eta(C_{k})= 2  $ by applying Lemma $ 2.8$ repeatedly, thus $ \eta(\Gamma)= 3= 2c(\Gamma)-1$. 

\textbf{ Case 2.} $ \Gamma \cong \theta(l^{'}, x^{'}, k^{'}) $, where $ \eta(C_{l^{'}+x^{'}})= \eta(C_{x^{'}+k^{'}})= 2 $  and $ l^{'}, x^{'}, k^{'} $ are even. (See fig. 1)

\textbf{ Case 2.1.} $ \sigma(C_{l^{'}+x^{'}})= \sigma(C_{k^{'}+x^{'}})= + $, then $ l^{'}+x^{'} \equiv 0( mod \ 4) $ and $ k^{'}+x^{'} \equiv 0( mod \ 4) $. If $ l^{'} $ is odd, then $ k^{'} $, $ x^{'} $ are odd. By Lemma $ 2.3 $, we have $ \eta(\Gamma) \leqslant \eta(\Gamma - u)+ 1 $, thus $ \eta(\Gamma - u) \geqslant 2 $, whereas $ \eta(\Gamma - u) =0 $ by applying Lemma $ 2.8$ repeatedly, a contradiction, so $ l^{'}, x^{'}, k^{'} $ are even.

If $ l^{'} \equiv 0( mod \ 4) $, then $ k^{'}=x^{'} \equiv 0( mod \ 4) $, by applying Lemma $ 4.1 $ repeatedly, we have $ \eta(\Gamma)= \eta(\theta(4, 4, 4)) $, where each cycle of $ \theta(4, 4, 4) $ is positive. By calculations, we can easily konw $ \eta(\theta(4, 4, 4))= 3 $. If $ l^{'} \equiv 2( mod \ 4) $, then $ k^{'}=x^{'} \equiv 2( mod \ 4) $, the proof is similar to the foregoing. 

As for $ \sigma(C_{l^{'}+x^{'}})= \sigma(C_{k^{'}+x^{'}})= - $ or $ \sigma(C_{k^{'}+x^{'}})= - $ and $ \sigma(C_{l^{'}+x^{'}})= + $, the proof is similar to case $ 2.1 $.

According to above discussion, the conclusion holds. 
\end{proof}

By slightly modifying the proofs of  Lemma $ 3.8 $ in [15] and Lemma $ 3.4 $ in [6], we have the following Lemmas for signed graph.

\begin{lemma}
	Let $\Gamma=\theta(k, l, m)$  be a bicyclic signed graph such that 
	the nullity of each cycle of $ \Gamma $ is $ 2 $ and 
	$k, l, m $ are even, then $\eta(\Gamma-v)=2$ for any vertex $ v $ in $ \Gamma $.

\end{lemma}
\begin{lemma}
Let $\Gamma$ be a connected leaf-free signed graph with $c(\Gamma) \geq 3$. If $\eta(\Gamma)=2 c(\Gamma)-1$, then

$ 	(1) $ $ \Gamma $ contains a cut-vertex and a pendant cycle.
	
$ 	(2)$ Let $\mathrm{C_{t}}$ be a pendant cycle and $H_{1}$ be the maximal leaf-free subgraph of $\Gamma-C_{t}$. Then the internal path connecting $\Gamma$ and $H_{1}$ is odd.
\end{lemma}

Secondly, we give the characterization of the connected leaf-free signed graph with  nullity $ 2 c(\Gamma)-1 $ and $c(\Gamma) \geq 3$.

\begin{Theorem}
	Let $\Gamma$ be a connected leaf-free signed graph with $c(\Gamma) \geq 3$. Then $\eta(\Gamma)=2 c(\Gamma)-1$ if and only if $\Gamma$ is a signed graph which is obtained from a tree $T$ with $p(T)=c(\Gamma)$ and $\eta(T)=p(T)-1$ by attaching a cycle with nullity $ 2 $ on each leaf of $T$.
\end{Theorem}
\begin{proof}
	We first prove the sufficiency part. Let $T$ be a tree with $p(T)=c(\Gamma)$ and $\eta(T)=p(T)-1$. Let $v_{1}, \cdots, v_{p(T)}$ be all leaves of $T$. $\Gamma$ is a signed graph obtained from $T$ by attaching a cycle with nullity $ 2 $, say $C_{i}$, on each $v_{i}(i \in\{1, \cdots, p(T)\})$. Firstly, we replace each $C_{i}$ of $\Gamma$ by a quadrangle $C_{4}^{i}$ with nullity $ 2 $, and denote the resulting signed graph by $\Gamma^{*}$. Then by Lemma $4.1$, we have $\eta(\Gamma)=\eta\left(\Gamma^{*}\right)$. Secondly, let $\Gamma^{**}$ be a signed graph which is obtained from $\Gamma^{*}$ by removing one vertex of each $C_{4}^{i},(i \in\{1, \cdots, p(T)\})$ which is adjacent to $v_{i}$. Then by Lemma $ 2.7$, $r\left(\Gamma^{*}\right)=r\left(\Gamma^{* *}\right)$. Moreover, by applying Lemma $ 2.8$ several times, we have $\eta\left(\Gamma^{* *}\right)=\eta(T)=p(T)-1=c(\Gamma)-1$. Therefore,
	$$
	\begin{aligned}
	r\left(\Gamma^{*}\right) &=r\left(\Gamma^{**}\right) \\
	n\left(\Gamma^{*}\right)-\eta\left(\Gamma^{*}\right) &=n\left(\Gamma^{* *}\right)-\eta\left(\Gamma^{**}\right) \\
	n\left(\Gamma^{*}\right)-\eta(\Gamma) &=\left(n\left(\Gamma^{*}\right)-c(\Gamma)\right)-(c(\Gamma)-1) \\
	\eta(\Gamma) &=2 c(\Gamma)-1 .
	\end{aligned}
	$$
	
	Now we will prove the necessity part. By Lemma $ 5.3(1) $ and Corollary $ 5.1 $, we know $ \Gamma $ contains a pendant cycle $ C $ with the nullity $ 2 $.
	
	Next, we prove the conclusion by induction on $c(\Gamma)$. 
	
	When $c(\Gamma)=3$, if there is only one pendant cycle in $\Gamma$, say $C_{1}$, then $\Gamma$ can be regarded as a signed graph obtained from $C_{1}$ and $\theta \cong \theta(l, x, k)$ by connecting $v_{1} \in V\left(C_{1}\right)$ and $v_{s} \in V(\theta)$ with a path $P\left(v_{1}, \cdots, v_{s}\right), s \geq 1$ (See Fig. 2).
	
	\begin{figure}[!h]
		\centering
		\includegraphics{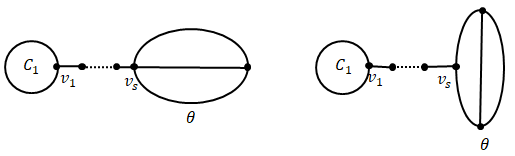}\\[1mm]
		{\small{  \bf Fig. 2~~}Leaf-free tricyclic signed graph with only one pendant cycle }
		\label{fig1}
	\end{figure}
	
	By Lemma $5.3(2) $, $s-1$ is odd, then we have $\eta\left(C_{1}\right)=\eta\left(C_{1}+\left\{v_{2}, \cdots, v_{s-1}\right\}\right)$ and $\eta\left(C_{1}-v_{1}\right)=\eta\left(\left(C_{1}-\right.\right.$ $\left.\left.v_{1}\right)+\left\{v_{1}, \cdots, v_{s}\right\}\right)=\eta\left(C_{1}+\left\{v_{2}, \cdots, v_{s}\right\}\right)$ by applying Lemma $ 2.8$ several times. Moreover, by Lemma $ 2.1 $ and $ 2.2 $, we have $ \eta\left(C_{1}\right)= \eta\left(C_{1}-v_{1}\right)+1 $. 
	Therefore, $v_{s}$ is a cut-vertex of $\Gamma$ such that $\eta\left(C_{1}+\left\{v_{2}, \cdots, v_{s-1}\right\}\right)=\eta\left(C_{1}+\left\{v_{2}, \cdots, v_{s}\right\}\right)+1$. By Lemma $ 2.4(2)$, we have
	$$
	\begin{aligned}
	\eta(\Gamma) &=\eta\left(\Gamma-v_{s}\right)-1 \\
	&=\eta\left(C_{1}+\left\{v_{2}, \cdots, v_{s-1}\right\}\right)+\eta\left(\theta-v_{s}\right)-1 \\
	&=\eta\left(C_{1}\right)+\eta\left(\theta-v_{s}\right)-1 \\
	&=2+\eta\left(\theta-v_{s}\right)-1 \\
	&=\eta\left(\theta-v_{s}\right)+1 .
	\end{aligned}
	$$
	
	When $\Gamma$ is the first graph in Fig. 2, clearly, $\theta-v_{s}$ is a tree with at most 3 leaves. By Theorem $3.1$, we have $\eta\left(\theta-v_{s}\right) \leq 3-1=2$. Then we have $\eta(\Gamma)=\eta\left(\theta-v_{s}\right)+1 \leq 2+1 \neq 5$, a contradiction.
	
	When $\Gamma$ is the second graph in Fig. 2, $\theta-v_{s}$ is a connected unicyclic signed graph and at most $ 2 $ leaves. By Theorem $ 3.1 $, we have $\eta\left(\theta-v_{s}\right) \leq 2 c\left(\theta-v_{s}\right)+2-1=3$. Then we have $\eta(\Gamma)=\eta\left(\theta-v_{s}\right)+1 \leq 3+1 \neq 5$, a contradiction.
	
	If there are exactly two pendant cycles in $\Gamma$, say $C_{1}$ and $C_{2}$, then $\Gamma$ can be regarded as a graph shown in top of Fig. 3. 
	
	\begin{figure}[!h]
		\centering
		\includegraphics{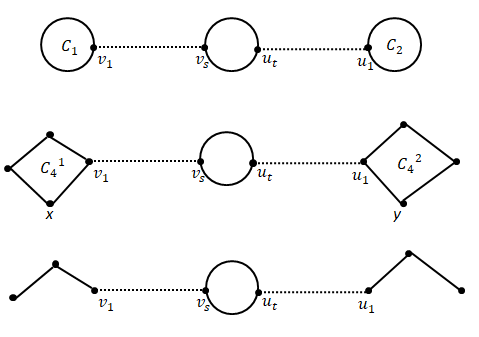}\\[1mm]
		{\small{  \bf Fig. 3~~} $\Gamma, \  \Gamma^{*} \  and \  \Gamma^{* *} $   }
		\label{fig1}
	\end{figure}
	
	By Corollary $ 5.1 $, $ \eta(C_{1})= \eta(C_{1})= 2 $. We replace $C_{i}(i=1,2)$ of $\Gamma$ by a quadrangle $C_{4}^{i}$, and denote the resulting signed graph shown in middle of Fig. 4 by $\Gamma^{*}$, where $ \eta(C_{4}^{i})= 2$. Then by Lemma $4.1$, we have $\eta(\Gamma)=\eta\left(\Gamma^{*}\right)$. Let $\Gamma^{**}$ be a signed graph which is obtained from $\Gamma^{*}$ by removing $x \in V\left(C_{4}^{1}\right)$ and $y \in V\left(C_{4}^{2}\right)$ (See third graph in Fig. 3). Then by Lemma $ 2.7$, $r\left(\Gamma^{*}\right)=r\left(\Gamma^{* *}\right)$. Therefore,
	$$
	\begin{aligned}
	\eta(\Gamma) &=\eta\left(\Gamma^{*}\right)=n\left(\Gamma^{*}\right)-r\left(\Gamma^{*}\right)=\left(n\left(\Gamma^{* *}\right)+2\right)-r\left(\Gamma^{* *}\right)=n\left(\Gamma^{* *}\right)+2. 
	\end{aligned}
	$$
	Then, $\eta\left(\Gamma^{* *}\right)=3$. However, by Lemma $ 5.3( 2 ) $, $ P\left(v_{1}, \cdots, v_{s}\right)$ and $P\left(u_{1}, \cdots, u_{t}\right) $ in $\Gamma^{* *}$ is odd, that is, $ t $ and $ s $ are even. By applying Lemma $ 2.8$ repeatedly, we have $\eta\left(\Gamma^{**}\right) \leq 2$, which leads to a contradiction. Therefore, all cycles of $\Gamma$ are pendant cycles.
	
	Let $C_{1}, C_{2}$ and $C_{3}$ be all pendant cycles of $ \Gamma $. By Corollary $ 5.1 $, $ \eta (C_{i})= 2(i=1,2,3)$. Similarly, we replace each $C_{i}(i=1,2,3)$ of $\Gamma$ by a quadrangle $C_{4}^{i}$ with the nullity $ 2 $, and denote the resulting graph by $\Gamma^{*}$. Then by Lemma $ 4.1 $, we have $\eta(\Gamma)=\eta\left(\Gamma^{*}\right)$. We continue to delete one vertex from each $C_{4}^{i}$ which is adjacent to a vertex with degree 3 and denote the resulting signed graph by $\Gamma^{**}$. We know that $\Gamma^{**}$ is a tree with 3 leaves. By Lemma $2.8, r\left(\Gamma^{*}\right)=r\left(\Gamma^{**}\right)$. Therefore,
	$$
	\begin{aligned}
	\eta(\Gamma) &=\eta\left(\Gamma^{*}\right) 
=n\left(\Gamma^{*}\right)-r\left(\Gamma^{*}\right) =\left(n\left(\Gamma^{* *}\right)+3\right)-r\left(\Gamma^{* *}\right)=\eta\left(\Gamma^{* *}\right)+3.
	\end{aligned}
	$$
	Then, $\eta\left(\Gamma^{**}\right)=\eta(\Gamma)-3=2$. It means that $\Gamma^{**}$ is a tree with $ 3 $ leaves and $\eta\left(\Gamma^{* *}\right)=2$. Thus, $\Gamma$ is a signed graph which is obtained from a tree $T$ with $p(T)=3$ and $\eta(T)=p(T)-1=2$ by attaching a cycle with nullity $ 2 $ on each leaf of $T$.
	
	 We assume that the result holds when $c(\Gamma)<c$. Now we consider the case $c(\Gamma)=c$ $(c \geq 4)$.
	
	Let $H$ be the maximal leaf-free subgraph of $\Gamma-C$. Then $\Gamma$ can be regarded as a signed graph which is obtained from $C$ and $H$ by connecting $v_{1} \in V(C)$ and $v_{l} \in V(H)$ with a path $P\left(v_{1}, v_{2}, \cdots, v_{l}\right)(l \geq 2)$, by Lemma $ 5.3(2) $, $ l $ is even (See Fig. 4).
	
		\begin{figure}[!h]
		\centering
		\includegraphics{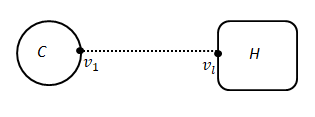}\\[1mm]
		{\small{  \bf Fig. 4~~} $\Gamma$   }
		\label{fig1}
	\end{figure}
	
Firstly, we have $\eta(H) \leq 2 c(H)-1$ by Theorem $ 3.1$. Secondly, we have $\eta(C)=\eta\left(C+\left\{v_{2}, \cdots, v_{l-1}\right\}\right)$ and $\eta\left(C-v_{1}\right)=\eta\left(\left(C-v_{1}\right)+\left\{v_{1}, \cdots, v_{l}\right\}\right)=\eta\left(C+\left\{v_{2}, \cdots, v_{l}\right\}\right)$ by applying Lemma $ 2.8$ several times. Therefore, $v_{l}$ is a cut-vertex of $\Gamma$ such that $\eta\left(C+\left\{v_{2}, \cdots, v_{l-1}\right\}\right)=\eta(C)=\eta\left(C-v_{1}\right)+1=\eta\left(C+\left\{v_{2}, \cdots, v_{1}\right\}\right)+1$. If $\eta(H) \leq 2 c(H)-2$, by Lemma $2.4(2)$, we have
	$$
	\begin{aligned}
	\eta(\Gamma) &=\eta\left(\Gamma-v_{l}\right)-1 \\
	&=\eta\left(C+\left\{v_{2}, \cdots, v_{l-1}\right\}\right)+\eta\left(H-v_{l}\right)-1 \\
	& \leq \eta(C)+(\eta(H)+1)-1 \\
	& \leq 2+(2 c(H)-2)+1-1 \\
	&=2(c(\Gamma)-1)=2 c(\Gamma)-2
	\end{aligned}
	$$
	a contradiction. Thus $\eta(H) \leq 2 c(H)-1$. 
	
	Since $c(H)=c(\Gamma)-1=c-1<c$, by the induction assumption, $H$ is a signed graph obtained from a tree $T^{\prime}$ with $p\left(T^{\prime}\right)=c-1$ and $\eta\left(T^{\prime}\right)=p\left(T^{\prime}\right)-1$ by attaching a cycle with nullity $ 2 $ on each leaf of $T^{\prime}$(See Fig. 5). It means that all $c-1$ cycles of $H$ are pendant cycles. Therefore, $\Gamma$ has at least $c-1$ pendant cycles.
	
		\begin{figure}[!h]
		\centering
		\includegraphics{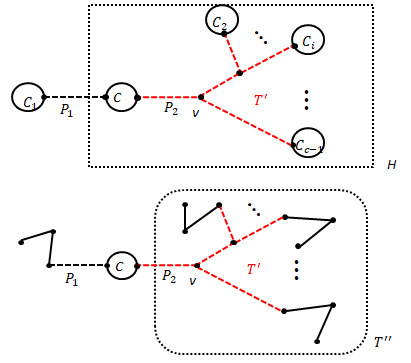}\\[1mm]
		{\small{  \bf Fig. 5                                                                                              ~~} $\Gamma$ and $ \Gamma^{* *} $  }
		\label{fig1}
	\end{figure}

	If there are exactly $c-1$ pendant cycles in $\Gamma$, then $\Gamma$ can be regarded as a signed graph in Fig. 5.
	Similarly, we replace $C_{\mathrm{i}}(i=1, \cdots, c-1)$ of $\Gamma$ by a quadrangle with nullity $ 2 $, the resulting signed graph is denoted by $ \Gamma^{*} $, then by Lemma $ 4.1 $, we have $ \eta(\Gamma^{*})= \eta(\Gamma) $. Removing one vertex from each quadrangle which is adjacent to a vertex with degree $ 3 $, the resulting graph is denoted by $\Gamma^{* *}$, then by Lemma $ 2.7$, we have $ r(\Gamma^{*})= r(\Gamma^{* *}) $. Then, $\eta\left(\Gamma^{* *}\right)=c$ by calculations. However, $\eta\left(\Gamma^{* *}\right)<c$. In fact, by Lemma $ 4.2(2) $, $v$ is a covered vertex in $T^{\prime \prime}$ and $\eta\left(T^{\prime \prime}\right)=p\left(T^{\prime \prime}\right)-1$. It follows from Lemma $ 2.6(1) $, we have $\eta\left(\Gamma^{* *}\right)=\eta\left(T^{\prime \prime}\right)+\eta\left(\Gamma^{* *}-T^{\prime \prime}\right)$. By Lemma $ 5.3(2) $, $  P_{1} $ and  $ P_{2} $ are odd in $ \Gamma^{* *}$. Then we have $\eta\left(\Gamma^{**}-T^{\prime \prime}\right)=1$ by applying Lemma $2.8$ repeatedly. Thus, $\eta\left(\Gamma^{**}\right)=\eta\left(T^{\prime \prime}\right)+\eta\left(\Gamma^{**}-T^{\prime \prime}\right)=\left(p\left(T^{\prime \prime}\right)-1\right)+1=p\left(T^{\prime \prime}\right)=c-2<c$.
	
	Therefore, all cycles of $\Gamma$ are pendant cycles, and the nullity of each cycle is $ 2 $. Similarly, we replace each cycle of $\Gamma$ by a quadrangle with nullity $ 2 $, and denote the resulting signed graph by $\Gamma^{*}$. Then by Lemma $4.1$, we have $\eta(\Gamma)=\eta\left(\Gamma^{*}\right)$. We continue to delete one vertex from each cycle which is adjacent to a vertex with degree $ 3 $ and the resulting graph is denoted by $\Gamma^{**}$. Clearly,  $\Gamma^{**}$ is a tree with $c$ leaves. By Lemma $2.7, r\left(\Gamma^{*}\right)=r\left(\Gamma^{**}\right)$. Therefore,
	$$
	\begin{aligned}
	\eta(\Gamma) &=\eta\left(\Gamma^{*}\right)=n\left(\Gamma^{*}\right)-r\left(\Gamma^{*}\right) =\left(n\left(\Gamma^{**}\right)+c\right)-r\left(\Gamma^{**}\right) =\eta\left(\Gamma^{**}\right)+c.
	\end{aligned}
	$$
	Then, $\eta\left(\Gamma^{**}\right)=\eta(\Gamma)-c=(2 c-1)-c=c-1$. It means that $\Gamma^{**}$ is a tree with $c$ leaves and $\eta\left(\Gamma^{**}\right)=c-1$. Thus, $\Gamma$ is a signed graph which is obtained from a tree $T$ with $p(T)=c$ and $\eta(T)=p(T)-1=c-1$ by attaching a cycle with nullity $ 2 $ on each leaf of $T$.
\end{proof}

Lastly, we give the characterization of the connected signed graph with nullity $ 2 c(\Gamma)+p(\Gamma)-1 $ and $c(\Gamma) \geq 1$.

\begin{Theorem}
Let $\Gamma$ be a connected signed graph with $c(\Gamma) \geq 1$. Then $\eta(\Gamma)=2 c(\Gamma)+p(\Gamma)-1$ if and only if $\Gamma$ has one of the following forms.

$ (1) $ A signed graph obtained from a tree $T$ with nullity $p(T)-1$ by attaching $c(\Gamma)$ cycles with nullity $ 2 $ on $c(\Gamma)$ leaves of $T$, where $p(T) \geq c(\Gamma)$.

$ (2) $ An $\infty$-graph $\infty(p, q, 1)$, where the nullity of each cycle of $\infty(p, q, 1)$ is $ 2 $, 

$ (3) $ $A\ \theta$-graph $\theta(p, q, l)$, where the nullity of each cycle of $\theta(p, q, l)$ is $ 2 $.	
\end{Theorem}
\begin{proof}
We first prove the sufficiency part. Signed graphs of form $ (2) $ or $ (3) $ have been proved in Theorem $ 5.2$. As for signed graphs of form $ (1) $, we proceed by induction on $c(\Gamma)$ to prove $\Gamma$ is $ 1- $deficient. When $c(\Gamma)=1$, it has been proved by Lemma $ 4.8 $. Suppose the result holds for connected signed graphs with $k$ elementary cycles, while $c(\Gamma)=k+1$. Let $C_{m}$ be a cycle with nullity $ 2 $ attached at a leaf of $T$. Let $ v $ be the vertex lying on $C_{m}$ and $ d_{\Gamma}(v)=3 $, so we have $ 1=\eta(C_{m}-v)=\eta(C_{m})-1 $ by Lemma $ 2.1 $ and $ 2.2 $, further by Lemma $2.4(1)$,  $ \eta(\Gamma)=\eta\left(\Gamma^{\prime}\right)+1$, where $\Gamma^{\prime}$ is the signed graph by shrinking $C_{m}$ into a vertex. Now $\Gamma^{\prime}$ is obtained by attaching $k$ cycles with nullity $ 2 $ at $k$ leaves of $T$, the induction hypothesis implies that $\Gamma^{\prime}$ is $ 1- $deficient. Observe that $p\left(\Gamma^{\prime}\right)=p(\Gamma)+1$, we have	
$$
\eta(\Gamma)=\eta\left(\Gamma^{\prime}\right)+1=\left[2 c\left(\Gamma^{\prime}\right)+p\left(\Gamma^{\prime}\right)-1\right]+1=2 c(\Gamma)+p(\Gamma)-1,
$$
which proves that $\Gamma$ is $ 1- $deficient.

Now we will prove the necessity part. Let $B$ be a block of $\Gamma$ such that $c(B) \geq c\left(B^{\prime}\right)$ for any block $B^{\prime}$ of $\Gamma$, then $B$ is $ 1- $deficient, otherwise, $\Gamma$ is $ 2^{+}- $deficient by Lemma $ 4.6$. Since $ B $ is $ 1- $deficient and  is a block, so $ \eta(B)=2c(B)-1 $, then by Lemma $5.3$, $1 \leq c(B) \leq 2$. We consider the following two cases depending on this.

\textbf{ Case 1.} $c(B)=2$.

Since $ B $ is a block of $\Gamma$, $B=\theta(k, l, m)$. By Lemma $ 4.10$, the nullity of each cycle is $ 2 $. Now we will prove $\Gamma=B$. Indeed, if $B$ is a proper subgraph of $\Gamma$, $ \Gamma$ has a cut-vertex $x$ which lies on $ B $. Note that we have $\eta(B-x)=2$ by Lemma $ 5.2$. Let $Q^{\prime}$ be a component of $\Gamma-x$ such that $V\left(Q^{\prime}\right) \cap V(B)=\emptyset$ and let $Q=Q^{\prime}+x$. Now, let $\Gamma^{\prime}$ be the subgraph induced by the vertices of $B$ and $Q$. If we can prove $\Gamma^{\prime}$ is $2^{+}$-deficient, then so is $\Gamma$ (by Lemma $ 4.6 $), and thus we derive a contradiction. Since $ 2=\eta(B-x)=\eta(B)-1 $ and $ x $ is a cut-vertex of $\Gamma^{\prime}$, by Lemma $2.4(1)$, we have $\eta\left(\Gamma^{\prime}\right)=\eta(Q)+\eta(B-x)=\eta(Q)+2$. If $Q$ is a cycle, then $\eta(Q) \leq 2$ by Lemma $ 2.2 $ and $ c\left(\Gamma^{\prime}\right)=3 $, thus $\eta\left(\Gamma^{\prime}\right) \leq 4=2 c\left(\Gamma^{\prime}\right)+p\left(\Gamma^{\prime}\right)-2$. Otherwise, according to  $\eta(Q) \leq 2 c(Q)+p(Q)-1$, $c(Q)=c\left(\Gamma^{\prime}\right)-2$, and $p(Q) \leq p\left(\Gamma^{\prime}\right)+1$ we have
$$
\eta\left(\Gamma^{\prime}\right)=\eta(Q)+2 \leq[2 c(Q)+p(Q)-1]+2 \leq 2 c\left(\Gamma^{\prime}\right)+p\left(\Gamma^{\prime}\right)-2.
$$
Hence, $\Gamma=B=\theta(k, l, m)$, and the nullity of each cycle is $ 2 $.

\textbf{ Case 2.} $c(B)=1$.

In this case, each block of $\Gamma$ contains at most one cycle. Therefore any two cycles share at most one common vertex, otherwise there exists a block with at least $ 2 $ cycles. If there are two cycles sharing a common vertex, then by Lemma $4.9$ we know $\Gamma$ has form $ (2) $.

Next, we assume that $\Gamma$ contains $c(\Gamma)$ vertex-disjoint cycles. By Lemma $ 4.10 $, the nullity of each cycle of $\Gamma$ is $ 2 $. Now we proceed by induction on $c(\Gamma)$ to prove that $\Gamma$ has form $(1)$. If $c(\Gamma)=1$, the result has been proved by Lemma $ 4.8$. Suppose the result holds for signed graphs with $k(\geq 1)$ disjoint cycles, while $c(\Gamma)=k+1$. Let $C_{m}$ be one of the cycles with nullity $ 2 $ and $\Gamma^{\prime}$ be the signed graph obtained by shrinking the cycle into a vertex. Since $\eta\left(\Gamma^{\prime}\right)=\eta(\Gamma)-1$ (by Lemma $ 2.4(1)$), $c\left(\Gamma^{\prime}\right)=c(\Gamma)-1$, and $p\left(\Gamma^{\prime}\right)=p(\Gamma)+1$, we have 
	$$
\begin{aligned}
\eta(\Gamma^{\prime}) &=\eta\left(\Gamma\right)-1 \\
&=2c\left(\Gamma\right)+p(\Gamma)-1-1 \\
&=2(c\left(\Gamma^{\prime}\right)+1)+(p\left(\Gamma^{\prime}\right)-1)-2\\
&=2c(\Gamma^{\prime})+p\left(\Gamma^{\prime}\right)-1,
\end{aligned}
$$
that is, $\Gamma^{\prime}$ is also $ 1- $deficient. The induction hypothesis implies that $\Gamma^{\prime}$ is obtained from a tree $T$ with $\eta(T)=p(T)-1$ by attaching $k$ pendant cycles with nullity $ 2 $ at $k$ leaves of $T$. Thus $\Gamma$ is obtained from $T$ by attaching $k+1$ pendant cycles with nullity $ 2 $ at $k+1$ leaves of $T$, where $c(\Gamma)=k+1 \leq p(T)$.
\end{proof}


\end{document}